\documentclass[10pt]{amsart}
\usepackage[left=3cm, right=3cm]{geometry}

\def\Escr{{\mathcal E}}
\def\Fscr{{\mathcal F}}
\def\Gscr{{\mathcal G}}
\def\Hscr{{\mathcal H}}

\def\Kscr{{\mathcal K}}

\def\Pscr{{\mathcal P}}
\def\Qscr{{\mathcal Q}}

\def\Ol{{\mathcal O}}

\def\id{\operatorname{id}}

\def\Hom{\operatorname {Hom}}

\def\RHomscr{\operatorname {\mathcal{R}\mathit{Hom}}}

\def\ker{\operatorname {Ker}}

\def\Tor{\operatorname {Tor}}

\def\coh{\operatorname {coh}}
\def\Qch{\operatorname {Qch}}

\usepackage{amsmath}
\usepackage{amssymb, amsfonts,amsthm,amscd}
\usepackage{latexsym,amsthm,amscd,flafter,graphicx,verbatim,stmaryrd,mathrsfs,color}
\usepackage[all]{xy}
\usepackage{mathrsfs}
\usepackage{nicefrac}
\usepackage{hyperref}

\newtheorem{lemma}{Lemma}[section]
\newtheorem{definition}[lemma]{Definition}
\newtheorem{proposition}[lemma]{Proposition}
\newtheorem{theorem}[lemma]{Theorem}
\newtheorem{corollary}[lemma]{Corollary}
\newtheorem{example}[lemma]{Example}

\newtheorem{remark}[lemma]{Remark}

\title[]{Adjoints to a Fourier-Mukai transform}
\author{\bf{Alice Rizzardo} }
\email[Alice Rizzardo]{alice.rizzardo@gmail.com}
\address{School of Mathematics\\The University of Edinburgh\\James Clerk Maxwell Building\\The King's Buildings\\Peter Guthrie Tait Road\\Edinburgh, EH9 3FD\\Scotland, UK}
\thanks{The author is a Postdoctoral Research Fellow at the University of Edinburgh.}

\begin{document}\begin{abstract}
Given a Fourier-Mukai functor $\Phi$ in the general setting of singular schemes, under various hypotheses we provide both left and a right adjoints to $\Phi$, and also give explicit formulas for them.  These formulas are simple and natural, and recover the usual formulas when the Fourier-Mukai kernel is a perfect complex. This extends previous work of \cite{tim, spaniards1, spaniards2} and has applications to the twist autoequivalences of \cite{DW}.
\end{abstract}

\maketitle

\section{Introduction}
Much of the current interest in derived categories, at least from a geometric perspective, began with the seminal result of Orlov \cite{Orlov}, who showed in 1997 that any fully faithful functor between the bounded derived categories of two smooth projective varieties $X$ and
$Y$ is isomorphic to a Fourier-Mukai transform, i.e. of the form
$$ \Phi^{X\to Y}_{\Pscr}(\cdot)= R\pi_{Y*}(L\pi_X^*(\cdot) \stackrel{L}{\otimes} \Pscr),$$
where $\pi_X$ and $\pi_Y$ are projections from $X\times Y$ to $X$ and $Y$ respectively.

More recently, there have been many attempts to improve our understanding of the Bridgeland stability manifold \cite{Bridgeland, Bayer-Bridgeland} and of the autoequivalence group of a derived category  \cite{autoequivalences, Toda}. To construct such autoequivalences, it is convenient to have explicit adjoints; however, for a Fourier-Mukai functor this is an open problem in a general setting. 

When $X$ and $Y$ are smooth over a field, and the kernel of the Fourier-Mukai transform is a perfect complex, it is well known that the left and right adjoints exist and have a prescribed form. This result has been extended by \cite{tim} to the case of separable schemes of finite type over a field. Further, \cite{spaniards1} gives a formula for adjoints for projective Gorenstein schemes over a field, under the weaker assumption that the kernel be of ``finite homological dimension'' over $X$ and $Y$, and \cite{spaniards2} gives the right adjoint in the  projective case under the same hypothesis for the kernel.

In this paper we prove the existence of both left and right adjoint to a Fourier-Mukai transform, and provide explicit formulas, under considerably weaker hypotheses. Our main result is the following:

\begin{theorem}[Theorem \ref{left}, Theorem \ref{right}]\label{thm1}
Let $X$ and $Y$ be separated, quasi-compact schemes, essentially of finite type over a Noetherian scheme $S$.
\begin{enumerate}
\item Let either $Y$ be quasi-projective over $S$, or $Y$ be such that there exists a compactification $Y\xrightarrow{j} \bar{Y}$ with $X\times_S \bar{Y}$ of finite Tor-dimension over $X$. Let $\Pscr$ be a $\pi_X$-perfect complex in $D^b_{\coh}(X\times_S Y)$. Then $\Phi^{X\to Y}_{\Pscr}:D_{\Qch}(X)\to D_{\Qch}(Y)$ has a left adjoint $\Phi^{Y\to X}_{\Pscr_L}$, where 
$$\Pscr_L=\RHomscr_{X\times_S Y}(\Pscr,\pi_X^!\Ol_X).$$

\item Let either $X$ be quasi-projective over $S$, or $X$ be such that there exists a compactification $X\xrightarrow{j} \bar{X}$ with $\bar{X}\times_S Y$ of finite Tor-dimension over $Y$. Let $\Pscr$ be a $\pi_Y$-perfect complex in $D^b_{\coh}(X\times_S Y)$. Then $\Phi^{X\to Y}_{\Pscr}:D_{\Qch}(X)\to D_{\Qch}(Y) $ has a right adjoint $\Phi^{Y\to X}_{\Pscr_R}$, where 
$$\Pscr_R=\RHomscr_{X\times_S Y}(\Pscr,\pi_Y^!\Ol_Y).$$
\end{enumerate}
\end{theorem}

The formulas for the adjoints in the theorem above coincide with the usual formulas when the kernel $\Pscr\in D^b_{\coh}(X\times_S Y)$ is a perfect complex. In general, however, dualizing is not well-behaved and therefore the formulas stated above are more natural and can be applied in this more general context.

When the given Fourier-Mukai transform is an equivalence of categories, we show that $\Pscr$ is both $\pi_X$- and $\pi_Y$-perfect, so that the explicit formulas above do realize the adjoints: 

\begin{corollary}
Assume that $X\to S$ and $Y\to S$ are projective. Let $\Pscr \in D^b_{\Qch}(X\times_S Y)$. If the functor $\Phi^{X\to Y}_{\Pscr}$ gives an equivalence $D^b_{\coh}(X)$ to $D^b_{\coh}(Y)$, then
\begin{enumerate}
\item $\Pscr$ is $\pi_X$- and $\pi_Y$-perfect
\item $\Phi^{X\to Y}_{\Pscr}$ has left and right adjoints given by the formulas of Theorem \ref{thm1}.
\end{enumerate}
\end{corollary}

One of the ingredients in this paper consists of recent results on the twisted inverse image pseudofunctor developed by Lipman \cite{Lipman}, Neeman \cite{Neeman} and their co-authors. These results hold on the level of the derived category of quasi-coherent sheaves, so our statements on the adjoints will hold in $D_{\Qch}(X)$ as opposed to the whole derived category of sheaves $D(X)$.

The results in this paper have applications to the study of flopping contractions and noncommutative deformations of Donovan-Wemyss, see \cite[Remark 7.7]{DW}.

\subsection{Acknowledgements}
This research was made possible in part by EPSRC grant number EP/K021400/1. The author would like to thank Michael Wemyss for many useful discussions. She would also like to thank Johan de Jong, Michel Van den Bergh, Alexander Kuznetsov, Julian Holstein, Greg Stevenson, and Long Dao for helpful conversations.

\subsection{Notation and conventions}
Throughout this paper, $X$ and $Y$ will be separated, quasi-compact schemes, essentially of finite type over a Noetherian scheme $S$. This implies that $X$ and $Y$ are Noetherian and so is $X\times_S Y$.

We denote by $D(X)$ the derived category of $\Ol_X$-modules. The subcategories $D^+(X)$, resp. $D^-(X)$ are the full subcategories whose objects are complexes whose cohomology vanishes in all but finitely many negative, resp. positive degrees; we set $D^b(X)=D^+(X)\cap D^-(X)$. For $\bullet= \emptyset, +,-,b$ we denote by $D^{\bullet}_{\Qch}(X)$, resp. $D^{\bullet}_{\coh}(X)$ the full subcategories whose objects are complexes with quasi-coherent, resp. coherent cohomology. 

We will omit the R's and L's in front of derived functors for simplicity. All functors are derived. We denote by $\pi_X$ and $\pi_Y$ the projections of the fiber product $X\times_S Y$ onto its factors. For $\Pscr\in D^b_{\coh}(X\times_S Y)$, the Fourier-Mukai transform with kernel $\Pscr$ is defined as the functor
$$\Phi^{X\to Y}_{\Pscr} (\cdot)=\pi_{Y*}(\pi_X^*(\cdot)\otimes \Pscr)$$
mapping $D_{\Qch}(X)$ to $D_{\Qch}(Y)$. 

\section{Preliminaries}

For the reader's convenience, we will state here some well-known fact about functors in derived categories.

\begin{lemma} \cite[Theorem A]{Spaltenstein}\label{Spaltenstein}
If $X$ is a scheme, and $\Fscr,\Gscr,\Hscr\in D(X)$, then there is a functorial isomorphism
$$\RHomscr_X(\Fscr,\RHomscr_X(\Gscr,\Hscr))\cong \RHomscr_X(\Fscr\otimes \Gscr,\Hscr).$$
For a morphism of schemes $f:X\to Y$, $\Fscr\in D(Y)$, $\Gscr\in D(X)$, we have
$$\RHomscr_Y(\Fscr, f_* \Gscr)\cong f_* \RHomscr_X(f^* \Fscr, \Gscr)$$
\end{lemma}

We collect in the following lemma some basic facts about open immersions:

\begin{lemma}\label{tim} 
Let $X$ be quasi-compact, quasi-separated and let $U\xrightarrow{j} X$ an open immersion. Then
\begin{enumerate}
\item $j_* \RHomscr_{U}(\Fscr,\Gscr) = \RHomscr_{X} (j_* \Fscr,j_*\Gscr)$ for all $\Fscr$, $\Gscr \in D(U)$
\item $j_*(\Fscr \otimes \Gscr)=j_*\Fscr\otimes j_*\Gscr$ for all $\Fscr$, $\Gscr \in D(U)$
\item $j^* \RHomscr_{X}(\Fscr,\Gscr) = \RHomscr_{U} (j^* \Fscr,j^*\Gscr)$ for all $\Fscr$, $\Gscr \in D(X)$
\end{enumerate}
\end{lemma}

\begin{proof}
The first two identities are well known and follow from the fact that for an open immersion we have an isomorphism $j^* j_*\cong \id$. For example, to see (1) observe that
$$j_*\RHomscr_{U}(\Fscr,\Gscr)=j_* \RHomscr_{U}(j^*j_* \Fscr,\Gscr)=\RHomscr_{X}(j_*\Fscr, j_*\Gscr)$$
by Lemma \ref{Spaltenstein}. The second assertion can be proved in a similar way.

The third identity is \cite[Proposition 4.6.7]{Lipman}.
\end{proof}

\subsection{The twisted inverse image pseudofunctor}
We use the notaton from \cite{Lipman}. By \cite[Theorem 1.2]{Neeman}, for a separated morphism $f:X\to Y$  essentially of finite type, there exists a well-defined functor $f^!: D_{\Qch}(Y) \to D_{\Qch}(X)$. If $f:X\to Y$ is proper then $f^!$ coincides with $f^{\times}$, the right adjoint to the derived direct image functor $Rf_*$. If $j$ is an open immersion then $j^!=j^*$. For the composition of two separated morphisms $X\xrightarrow{f} Y\xrightarrow{g} Z$  essentially of finite type, we have that $(gf)^!=f^!g^!$ if one of the conditions below holds:
\begin{itemize}
\item $f$ is of finite Tor-dimension.
\item The composite $gf$ is proper.
\item We restrict to the subcategory $D^+_{\Qch}(Z)\subset D_{\Qch}(Z)$.
\end{itemize}

\begin{lemma}\label{ref-4.4-2}\cite[Proposition 7.11]{Neeman}
If $f : X\to Y$ is a separated morphism essentially of finite type, $E$ is a perfect complex in $D_{\Qch}(Y )$ and $F\in D_{\Qch}(Y )$ is arbitrary, then the map
$$\sigma(f, E, F) : f^*E\otimes f^!F\to f^!(E\otimes F)$$
is an isomorphism. If f is of finite Tor-dimension then $\sigma(f, E, F)$ is an isomorphism for all $E,F\in D_{\Qch}(Y)$.
\end{lemma}


\subsection{Relative perfect complexes}

We remind the reader of the following definition:
\begin{definition} \cite[Definition 4.1]{Illusie}
Let $f:X\to Y$ be a morphism. A complex $F\in D^b_{\coh}(X)$ is said to be \emph{perfect relative to $f$} if it is locally isomorphic in the derived category of $f^{-1}\Ol_Y$-modules to a bounded complex of flat $f^{-1}\Ol_Y$-modules.
\end{definition}
In particular, $f$ is perfect if $\Ol_X$ is perfect over $f$. Perfection over $\id_X$ is equivalent to perfection in $D(X)$. Moreover, if $f$ is a perfect morphism then a perfect complex in $D^b_{\coh}(X)$ is $f$-perfect (see \cite[Corollaire 4.5.1]{Illusie}). More generally, if $f$ is locally of finite type and finite Tor-dimension and $Y$ is Noetherian, then by \cite[Example 1.3]{Illusie} and \cite[Corollaire 4.5.1]{Illusie} a perfect complex in $D^b_{\coh}(X)$ is perfect with respect to $f$.

\begin{remark}
This notion satisfies two-out-of-three (see \cite[Notation 4.2]{Illusie}).
\end{remark}


\begin{lemma}\cite[Corollaire 3.4]{Illusie} \label{tensor with perfect}
Let $f:X\to Y$ be a morphism of schemes. If $\Fscr$ is a perfect complex on $X$ and $\Escr$ is perfect with respect to $f$, then $\Fscr\otimes \Escr$ is perfect with respect to $f$.
\end{lemma}

Before we continue, let us give some examples of complexes that are perfect relative to a morphism. We will use the following:

\begin{proposition}\cite[Corollaire 3.7.1, Proposition 4.8]{Illusie}\label{diagram}
Let $S$ be a locally Noetherian scheme, and consider a diagram of schemes
$$\xymatrix{
X \ar[rr]^f \ar[dr]_g && Y \ar[dl]^h\\
& S
}$$
Then
\begin{enumerate}
\item If $f$ is a quasi-compact, quasi-separated morphism of schemes that are quasi-compact over $S$, and $E\in D^b_{\Qch}(X)$ is perfect with respect to $g$, then $f_* E$ is perfect with respect to $h$.
\item If $f$ is proper, $X$ and $Y$ are locally of finite type over $Z$, and $E\in D^b(X)$ is perfect with respect to $g$, then $f_* E$ is perfect with respect to $h$.
\end{enumerate}

with $f$ proper, $X$ and $Y$ locally of finite type over $Z$. If $E\in D^b(X)$ is perfect with respect to $g$ then $f_* E$ is perfect with respect to $h$.
\end{proposition}

\begin{example}
If $X$ is a scheme over $S$, $S$ is Noetherian and $X$ is locally of finite type over $S$, then $\Ol_{\Delta}\in D^b(X\times_S X)$ is $\pi_X$-perfect, where $\pi_X$ is the projection from $X\times_S X$ to one of the factors. This follows from Proposition \ref{diagram} with $f:X\to X\times_S X$ being the diagonal embedding, $Z=X$.
\end{example}

\begin{example}
If $C\xrightarrow{i} X\times_S X$ is a closed immersion, again by Proposition \ref{diagram} we have that $i_* \Ol_C$ is $\pi_X$-perfect if the composition $\pi_X\circ i: C\to X$ is a perfect morphism.
\end{example}

Perfect complexes and $f$-perfect complexes behave well under $\RHomscr$:

\begin{proposition}\cite[Proposition 2.3.9]{reflexivity}\label{perfect is mantained}
Let $f:X\to Y$ be a morphism of schemes. If $\Escr$ is a perfect complex on $Y$ then $\RHomscr(-,f^!\Escr)$ takes $f$-perfect complexes to $f$-perfect complexes.
\end{proposition}

\begin{definition}\cite[Definition 1.3.1]{reflexivity}
Given $\Fscr\in D(X)$ and $\Pscr\in D^b_{\coh}(X)$, we say that $\Pscr$ is derived $\Fscr$-reflexive if $\RHomscr(\Pscr,\Fscr)\in D^b_{\coh}(X)$ and the canonical morphism
$$\delta_{\Fscr}^{\Pscr}: \Pscr \to \RHomscr_X(\RHomscr_X(\Pscr,\Fscr),\Fscr)$$
corresponding via Lemma \ref{Spaltenstein} to the natural composition
$$\Pscr\otimes_X \RHomscr_X(\Pscr,\Fscr)\xrightarrow{\cong}\RHomscr_X(\Pscr,\Fscr)\otimes_X \Pscr \xrightarrow{\epsilon} \Fscr,$$
where $\epsilon$ is the evaluation map, is an isomorphism.
\end{definition}

\begin{lemma}\label{reflexive} \cite[Corollaire 4.9.2]{Illusie}
If $f:X\to Y$ is a compactifiable morphism, then each complex that is perfect over $f$ is derived $f^!\Ol_Y$-reflexive.
\end{lemma}
In particular, by \cite{Nagata} this is true for a separated morphism of finite type between quasi-compact quasi-separated schemes.

\begin{lemma}\label{ref-4.1-1} \cite[Lemma 6.6]{reduction}
Let $X$, $Y$ be Noetherian schemes, $g:X\to Y$ be a morphism of finite Tor dimension. Let $\Pscr$ be a $g$-perfect complex. Let $\Escr$ a perfect complex over $Y$ and $\Fscr\in D_{\Qch}(Y)$. Then we have an isomorphism
$$\RHomscr_X(\Pscr,g^! \Escr)\otimes_X g^*\Fscr \cong \RHomscr_X(\Pscr,g^!\Escr\otimes g^*\Fscr)$$
\end{lemma}

In the given reference the authors give a proof under the more restrictive, but unnecessary, hypothesis that $\Fscr\in D_{\Qch}^+(Y)$. This is solely because the general technology was not available at the time. 

\subsection{The projective case}
In the projective case, some of the results we mentioned above are true without the hypothesis of finite Tor dimensionality. In particular, we will prove an analogue of Lemma \ref{ref-4.1-1} in Lemma \ref{pull out the tensor} below. This requires the following lemma:

\begin{lemma}\cite[Proposition 2.5]{spaniards2}, \cite[Lemma 1.4]{spaniards1} \label{beilinson}
Let $f:X\to Y$ a projective morphism and $\Ol(1)$ be a relatively very ample line bundle on $X$, let $\Escr,\Fscr \in D_{\Qch}(X)$. Then
\begin{enumerate}
\item $\Escr \in D^b_{\Qch}(X)$ if and only if $f_*(\Escr(n))\in D^b_{\Qch}(Y)$ for every integer $n$. 
\item  A morphism $g: \Escr\to \Fscr$ is an isomorphism if and only if the induced morphism $f_*(\Escr(n))\to f_*(\Fscr(n))$ is an isomorphism in $D_{\Qch}(Y)$ for every integer $n$.
\end{enumerate}
\end{lemma}


\begin{lemma}\cite[Proposition 2.8]{spaniards2}\label{pull out the tensor}
Let $g:X\to Y$ be a projective morphism. Let $\Pscr$ be an object in $D^b_{\coh}(X)$, perfect with respect to $Y$. Let $\Fscr\in D_{\Qch}(Y)$. Then we have an isomorphism
$$\RHomscr_X(\Pscr,g^! \Ol_Y)\otimes_X g^*\Fscr \cong \RHomscr_X(\Pscr,g^!\Fscr).$$
\end{lemma}

\begin{proof}
In the given reference, the concept of perfection with respect to a morphism is replaced by the concept of ``finite homological dimension'' (and the authors ask for $\Fscr\in D^b_{\coh}(Y)$). For the convenience of the reader we repeat their proof in our setting. 

We have natural morphisms
$$\RHomscr_X(\Pscr,g^! \Ol_Y)\otimes_X g^*\Fscr \to \RHomscr_X(\Pscr,g^! \Ol_Y\otimes_X g^*\Fscr )\to \RHomscr_X(\Pscr,g^!\Fscr),$$
where the first map is as in \cite[(5.5.2)]{reduction}, and the second map is defined as in \cite[Definition 7.7]{Neeman}. We have to prove that the composition is an isomorphism.  The question being local on $Y$, we can assume $g$ to be projective. By \cite[Lemma 2.5]{spaniards2}, it is enough to show that the induced morphism
$$g_*(\RHomscr_X(\Pscr,g^! \Ol_Y)\otimes_X g^*\Fscr\otimes \Ol_X(n)) \to g_*(\RHomscr_X(\Pscr,g^!\Fscr)\otimes \Ol_X(n))$$
is an isomorphism for all $n$. Note that 
\begin{align*}
\RHomscr_X(\Pscr,g^! \Ol_Y)\otimes_X g^*\Fscr\otimes \Ol_X(n) &= \RHomscr_X(\Pscr\otimes\Ol_X(-n),g^!\Ol_Y)\otimes g^*\Fscr\\
\RHomscr_X(\Pscr,g^!\Fscr)\otimes \Ol_X(n) &= \RHomscr_X(\Pscr\otimes \Ol_X(-n),g^!\Fscr)
\end{align*}
(see for example \cite[Lemma 1.4.6]{reflexivity}) because $\Ol_X(n)$ is perfect.

By \cite[Tag 0A9I]{stacks-project}, $g^!\Ol_Y\in D^+_{\Qch}(X)$, and hence by \cite[\href{http://stacks.math.columbia.edu/tag/0123}{Tag 0A6H}]{stacks-project} we have that $\RHomscr_X(\Pscr\otimes\Ol_X(-n),g^!\Ol_Y)\in D_{\Qch}(X)$. Hence by projection formula and Grothendieck duality (\cite[Lemma 3.2]{Neeman}) we have to show that 
$$\RHomscr_Y(g_*\Pscr(-n),\Ol_Y)\otimes_X \Fscr \to \RHomscr_X(g_*\Pscr(-n),\Fscr)$$
is an isomorphism for all $n$. Since by Lemma \ref{tensor with perfect} $\Pscr(-n)$ is perfect with respect to $g$, hence  $g_*\Pscr(-n)$ is perfect by \cite[Proposition 4.8]{Illusie}, this follows by \cite[Lemma 1.4.6]{reflexivity}.
\end{proof}

%
%

\section{The left adjoint}

\begin{theorem}\label{left}
Let either $Y$ be a quasi-projective scheme over $S$, or $Y$ be such that there exists a compactification $Y\xrightarrow{j} \bar{Y}$ with $X\times_S \bar{Y}$ of finite Tor-dimension over $X$. Let $\Pscr$ be a $\pi_X$-perfect complex in $D^b_{\coh}(X\times_S Y)$. Then 
$$\Phi^{X\to Y}_{\Pscr} : D_{\Qch}(X)\to D_{\Qch}(Y)$$ 
has a left adjoint 
$$\Phi^{Y\to X}_{\Pscr_L} : D_{\Qch}(Y) \to D_{\Qch}(X),$$ 
where 
$$\Pscr_L=\RHomscr_{X\times_S Y}(\Pscr,\pi_X^!\Ol_X).$$
\end{theorem}

\begin{remark}
The finite Tor-dimension condition is satisfied for example when $X$ and $Y$ are defined over a field, i.e. $S=\text{Spec}(k)$.
\end{remark}

\begin{proof} 
In the quasi-projective setting, let $Y\to S$ factor as $Y\xrightarrow{j}\bar{Y} \to S$ with $\bar{Y}$ projective over $S$ and $j$ an open immersion. In the finite Tor dimension setting, let $j$ and $\bar{Y}$ as in the statement of the theorem. Hence in either case we have a diagram
$$\xymatrix{
& X\times_S \bar{Y} \ar@/^1.6pc/[ddrr]^{\bar{\pi}_Y} \ar@/_1.3pc/[ddl]_{\bar{\pi}_X}\\
& X\times_S Y \ar[u]_{\id\times j} \ar@/^0.8pc/[dr]^{\pi_Y} \ar@/_0.7pc/[dl]_{\pi_X}\\
X  && Y \ar[r]_-{j} & \bar{Y}
}$$
where $j$ is an open immersion and $\bar{\pi}_X$ is either proper and finite Tor-dimension or a projective morphism (since proper and  projective morphisms are stable under base change). 

We have functorial isomorphisms
\begin{align}
\Hom_X&(\Phi^{Y\to X}_{\Pscr_L} (\Fscr), \Gscr)  \nonumber\\
=& \Hom_X(\pi_{X_*}(\Pscr_L\otimes \pi_Y^*\Fscr), \Gscr)  \nonumber\\
=& \Hom_X(\bar{\pi}_{X_*} (\id \times j)_* (\Pscr_L\otimes \pi_Y^*\Fscr), \Gscr) \tag{L2}\label{L2}  \\
=& \Hom_{X\times_S \bar{Y}}((\id \times j)_*(\Pscr_L\otimes \pi_Y^*\Fscr) , \bar{\pi}_X^! \Gscr) \tag{L3}\label{L3} \\
=& \Hom_{X\times_S \bar{Y}}((\id \times j)_*(\RHomscr_{X\times_S Y}(\Pscr,\pi_X^!\Ol_X)\otimes \pi_Y^*\Fscr) , \bar{\pi}_X^! \Gscr) \tag{L4}\label{L4} \\
=& \Hom_{X\times_S \bar{Y}}((\id \times j)_*(\RHomscr_{X\times_S Y}(\Pscr,(\id\times j)^*\bar{\pi}_X^!\Ol_X)\otimes \pi_Y^*\Fscr) , \bar{\pi}_X^! \Gscr) \tag{L5}\label{L5} \\
=& \Hom_{X\times_S \bar{Y}}((\id \times j)_*(\RHomscr_{X\times_S Y}((\id\times j)^* (\id\times j)_*\Pscr,(\id\times j)^*\bar{\pi}_X^!\Ol_X)\otimes \pi_Y^*\Fscr) , \bar{\pi}_X^! \Gscr) \tag{L6}\label{L6} \\
=& \Hom_{X\times_S \bar{Y}}((\id \times j)_*((\id\times j)^* \RHomscr_{X\times_S \bar{Y}}((\id\times j)_*\Pscr,\bar{\pi}_X^!\Ol_X)\otimes \pi_Y^*\Fscr) , \bar{\pi}_X^! \Gscr) \tag{L7}\label{L7} \\
=& \Hom_{X\times_S \bar{Y}}( \RHomscr_{X\times_S \bar{Y}}((\id\times j)_*\Pscr,\bar{\pi}_X^!\Ol_X)\otimes (\id \times j)_* \pi_Y^*\Fscr, \bar{\pi}_X^! \Gscr) \tag{L8}\label{L8} \\
=& \Hom_{X\times_S \bar{Y}}( (\id \times j)_* \pi_Y^*\Fscr , \RHomscr_{X\times_S \bar{Y}} (\RHomscr_{X\times_S \bar{Y}}((\id\times j)_*\Pscr,\bar{\pi}_X^!\Ol_X),\bar{\pi}_X^! \Gscr)) \tag{L9}\label{L9} 
\end{align}
where \eqref{L3} holds because $\bar{\pi}_{X*}$ is left adjoint to $\bar{\pi}_X^!$ since $\bar{\pi}_X$ is proper. The equality \eqref{L4} follows by the definition of $\Pscr_L$. The equality \eqref{L5} holds because $\pi_X^! = (\bar{\pi}_X\circ (\id\times j))^!= (\id\times j)^! \bar{\pi}_X^!$ since $\id\times j$ is finite Tor dimension. Equality \ref{L6} holds because $(\id\times j)^* (\id\times j)_*=\id$ by Lemma \ref{tim}. The equality \eqref{L7} holds by Lemma \ref{tim}. Then, \eqref{L8} is by projection formula and \eqref{L9} follows by Lemma \ref{Spaltenstein}. 

Now note that by Proposition \ref{diagram}(1), $(\id\times j)_* \Pscr$ is $\bar{\pi}_X$-perfect (an open immersion of Noetherian schemes is quasi-compact); hence by Proposition \ref{perfect is mantained}, $\RHomscr_{X\times_S \bar{Y}}((\id\times j)_*\Pscr,\bar{\pi}_X^!\Ol_X)$ is $\bar{\pi}_X$-perfect. Hence we can continue our chain of natural isomorphisms:

\begin{align}
=& \Hom_{X\times_S \bar{Y}}( (\id \times j)_* \pi_Y^*\Fscr , \RHomscr_{X\times_S \bar{Y}} (\RHomscr_{X\times_S \bar{Y}}((\id\times j)_*\Pscr,\bar{\pi}_X^!\Ol_X),\bar{\pi}_X^! \Ol_X)\otimes \bar{\pi}_X^* \Gscr) \tag{L10}\label{L10}\\ 
=& \Hom_{X\times_S \bar{Y}}( (\id \times j)_* \pi_Y^*\Fscr ,  (\id\times j)_*\Pscr\otimes \bar{\pi}_X^! \Gscr) \tag{L11}\label{L11}\\ 
=& \Hom_{X\times_S \bar{Y}}( (\id \times j)_* \pi_Y^*\Fscr , (\id\times j)_*(\Pscr\otimes (\id\times j)^*\bar{\pi}_X^* \Gscr)) \tag{L12}\label{L12}\\ 
=& \Hom_{X\times_S Y}( \pi_Y^*\Fscr , \Pscr\otimes (\id\times j)^*\bar{\pi}_X^* \Gscr) \tag{L13}\label{L13}\\ 
=& \Hom_{Y}( \Fscr,\pi_{Y*}( \Pscr \otimes \pi_X^*\Gscr)) \tag{L14}\label{L14}\\
=& \Hom_{Y}( \Fscr,  \Phi^{X\to Y}_{\Pscr} (\Gscr)) \tag{L15}\label{L15}.
\end{align}

The equality \eqref{L10} follows by Proposition \ref{pull out the tensor} in the projective case; in the finite Tor-dimension case, it follows because
\begin{align*}
\RHomscr_{X\times_S \bar{Y}} &(\RHomscr_{X\times_S \bar{Y}}((\id\times j)_*\Pscr,\bar{\pi}_X^!\Ol_X),\bar{\pi}_X^! \Gscr)\\
&=\RHomscr_{X\times_S \bar{Y}} (\RHomscr_{X\times_S \bar{Y}}((\id\times j)_*\Pscr,\bar{\pi}_X^!\Ol_X),\bar{\pi}_X^!\Ol_X\otimes \bar{\pi}_X^* \Gscr)\\
&= \RHomscr_{X\times_S \bar{Y}} (\RHomscr_{X\times_S \bar{Y}}((\id\times j)_*\Pscr,\bar{\pi}_X^!\Ol_X),\bar{\pi}_X^! \Ol_X)\otimes \bar{\pi}_X^* \Gscr
\end{align*}
by Lemma \ref{ref-4.4-2} and Lemma \ref{ref-4.1-1}. 

The equality \eqref{L11} is by Proposition \ref{reflexive}, and \eqref{L12} is by projection formula. Further, \eqref{L13} follows by Lemma \ref{tim}. We conclude with \eqref{L14} by adjunction of $\pi_Y^*$ and $\pi_{Y*}$.
\end{proof}

\section{The right adjoint}

\begin{lemma}\label{right}
Let either $X$ be a quasi-projective scheme over $S$, or $X$ be such that there exists a compactification $X\xrightarrow{j} \bar{X}$ with $\bar{X}\times_S Y$ of finite Tor-dimension over $Y$. Let $\Pscr$ be a $\pi_Y$-perfect complex in $D^b_{\coh}(X\times_S Y)$. Then the functor
$$\Phi^{X\to Y}_{\Pscr} :D_{\Qch}(X)\to D_{\Qch}(Y)$$
has a right adjoint 
$$\Phi^{Y\to X}_{\Pscr_R} : D_{\Qch}(Y)\to D_{\Qch}(X),$$
where 
$$\Pscr_R=\RHomscr_{X\times_S Y}(\Pscr,\pi_Y^!\Ol_Y).$$
\end{lemma}

\begin{proof}
In the quasi-projective setting, let $X\to S$ factor as $X\xrightarrow{j}\bar{X} \to S$ with $\bar{X}$ projective over $S$ and $j$ an open immersion. In the finite Tor dimension setting, let $j$ and $\bar{X}$ be as in the statement of the theorem. Hence in either case we have a diagram
$$\xymatrix{
&& \bar{X}\times_S Y \ar@/_1.6pc/[ddll]_{\bar{\pi}_X} \ar@/^1.3pc/[ddr]^{\bar{\pi}_Y}\\
&& X\times_S Y \ar[u]^{j\times \id} \ar@/_0.8pc/[dl]_{\pi_X} \ar@/^0.7pc/[dr]^{\pi_Y}\\
\bar{X} & X \ar[l]_-{j} && Y
}$$
where $j$ is an open immersion and $\bar{\pi}_Y$ is either proper and finite Tor-dimension or a projective morphism (since proper and  projective morphisms are stable under base change). We have functorial isomorphisms

\begin{align*}
\Hom_{X}&(\Gscr, \Phi^{Y\to X}_{\Pscr_R} (\Fscr)) \\
&= \Hom_{X}(\Gscr, \pi_{X_*}(\Pscr_R\otimes \pi_Y^*\Fscr))  \\
&= \Hom_{X\times_S Y} (\pi_X^* \Gscr, \Pscr_R\otimes \pi_Y^*\Fscr) \tag{R2}\label{R2} \\
&= \Hom_{\bar{X}\times_S Y} ((j\times \id)_* \pi_X^* \Gscr,(j\times\id)_* ( \Pscr_R\otimes \pi_Y^*\Fscr)) \tag{R3}\label{R3} \\
&= \Hom_{\bar{X}\times_S Y} ((j\times \id)_* \pi_X^* \Gscr,(j\times\id)_* \Pscr_R\otimes \bar{\pi}_Y^*\Fscr) \tag{R4}\label{R4} \\
\end{align*}
where \eqref{R2} is by $\pi_X^*$-$\pi_{X_*}$ adjunction, \eqref{R3} is by Lemma \ref{tim}, and \eqref{R4} is by projection formula. 

Now note that, by Proposition \ref{perfect is mantained}, $\Pscr_R$ is also perfect with respect to $\pi_Y$, and hence by Proposition \ref{diagram}(1), $(j\otimes \id)_* \Pscr_R$ is $\bar{\pi}_Y$-perfect (an open immersion of Noetherian schemes is quasi-compact). Hence we can continue our chain of natural isomorphisms:

\begin{align*}
&= \Hom_{\bar{X}\times_S Y} ((j\times \id)_* \pi_X^* \Gscr,\RHomscr_{\bar{X}\times_S Y}(\RHomscr_{\bar{X}\times_S Y} ((j\times\id)_* \Pscr_R, \bar{\pi}_Y^! \Ol_Y),\bar{\pi}_Y^!\Ol_Y) \otimes \bar{\pi}_Y^*\Fscr) \tag{R5}\label{R5} \\
&= \Hom_{\bar{X}\times_S Y} ((j\times \id)_* \pi_X^* \Gscr,\RHomscr_{\bar{X}\times_S Y}(\RHomscr_{\bar{X}\times_S Y} ((j\times\id)_* \Pscr_R, \bar{\pi}_Y^! \Ol_Y),\bar{\pi}_Y^! \Fscr)) \tag{R6}\label{R6} \\
&= \Hom_{\bar{X}\times_S Y} ((j\times \id)_* \pi_X^* \Gscr \otimes \RHomscr_{\bar{X}\times_S Y} ((j\times\id)_* \Pscr_R, \bar{\pi}_Y^! \Ol_Y),\bar{\pi}_Y^! \Fscr) \tag{R7}\label{R7} \\
&= \Hom_{\bar{X}\times_S Y} ((j\times \id)_* (\pi_X^* \Gscr \otimes (j\times \id)^* \RHomscr_{\bar{X}\times_S Y} ((j\times\id)_* \Pscr_R, \bar{\pi}_Y^! \Ol_Y)),\bar{\pi}_Y^! \Fscr) \tag{R8}\label{R8} \\
&= \Hom_{\bar{X}\times_S Y} ((j\times \id)_* (\pi_X^* \Gscr \otimes  \RHomscr_{X\times_S Y} ( \Pscr_R, \pi_Y^! \Ol_Y)),\bar{\pi}_Y^! \Fscr) \tag{R9}\label{R9} \\
&= \Hom_{\bar{X}\times_S Y} ((j\times \id)_* (\pi_X^* \Gscr \otimes  \Pscr) ,\bar{\pi}_Y^! \Fscr) \tag{R10}\label{R10} \\
&= \Hom_{Y} (\bar{\pi}_{Y*}(j\times \id)_* (\pi_X^* \Gscr \otimes \Pscr), \Fscr) \tag{R11}\label{R11}\\
&= \Hom_{Y} (\pi_{Y*}(\pi_X^*\Gscr\otimes \Pscr),\Fscr) \tag{R12}\label{R12}\\
&= \Hom_{Y} (\Phi^{X\to Y}_{\Pscr} (\Gscr),\Fscr) \tag{R13}\label{R13}.
\end{align*}

The equality \eqref{R5} follows by Proposition \ref{reflexive}. By Proposition \ref{perfect is mantained}, $\RHomscr_{\bar{X}\times_S Y} ((j\times\id)_* \Pscr_R, \bar{\pi}_Y^! \Ol_Y),\bar{\pi}_Y^!\Ol_Y) $ is $\bar{\pi}_Y$-perfect, hence equality \eqref{R6} follows by Proposition \ref{pull out the tensor} in the projective case; in the finite Tor-dimension case, it follows because
\begin{align*}
\RHomscr_{\bar{X}\times_S Y} & (\RHomscr_{\bar{X}\times_S Y} ((j\times\id)_* \Pscr_R, \bar{\pi}_Y^! \Ol_Y),\bar{\pi}_Y^!\Ol_Y) \otimes \bar{\pi}_Y^*\Fscr \\
&=\RHomscr_{\bar{X}\times_S Y} (\RHomscr_{\bar{X}\times_S Y} ((j\times\id)_* \Pscr_R, \bar{\pi}_Y^! \Ol_Y),\bar{\pi}_Y^!\Ol_Y \otimes \bar{\pi}_Y^*\Fscr)\\ 
&=\RHomscr_{\bar{X}\times_S Y}(\RHomscr_{\bar{X}\times_S Y} ((j\times\id)_* \Pscr_R, \bar{\pi}_Y^! \Ol_Y),\bar{\pi}_Y^! \Fscr)
\end{align*}
by Lemma \ref{ref-4.1-1} and Lemma \ref{ref-4.4-2}. Then, \eqref{R7} follows by Lemma \ref{Spaltenstein} and \eqref{R8} is by projection formula. Equality \eqref{R9} holds by Lemma \ref{tim} and \eqref{R10} is by definition of $\Pscr_R$ and Lemma \ref{reflexive}. Equality \eqref{R11} holds by adjunction of $\bar{\pi}_{Y*}$ and $\bar{\pi}_Y^!$ since $\bar{\pi}_Y$ is proper. 
\end{proof}

\section{The projective case}

In this last section we provide two corollaries to our main result in the case where the schemes $X$ and $Y$ are projective over the base $S$. In this case, from the behavior of the Fourier-Mukai transform $\Phi^{X\to Y}_{\Pscr}$ we can deduce properties of the kernel $\Pscr$:
\begin{enumerate}
\item $\Phi^{X\to Y}_{\Pscr}$ maps $D^b_{\coh}(X)$ to $D^b_{\coh}(Y)$ if and only if $\Pscr$ is $\pi_X$-perfect (Corollary \ref{X-perfect}).
\item If $\Phi^{X\to Y}_{\Pscr}$ is an equivalence, then $\Pscr$ is $\pi_X$- and $\pi_Y$-perfect (Corollary \ref{both_perfect}).
\end{enumerate}
Then in both cases we get a left adjoint as in Theorem \ref{left}; in the second case, we get an adjoint as in Theorem \ref{right}.

\begin{lemma}\label{finite homological dimension}
Let $\Pscr \in D^b_{\Qch}(X\times_S Y)$. The following are equivalent:
\begin{enumerate}
\item $\Pscr \in D^b_{\coh}(X\times_S Y)$ and it is $\pi_X$-perfect.
\item For all $\Fscr\in D^b_{\coh}(X)$ we have $\Pscr\otimes \pi_X^*\Fscr\in D^b_{\coh}(X\times_S Y)$.
\end{enumerate}
\end{lemma}

\begin{proof}
Assume that for all $\Fscr\in D^b_{\coh}(X)$ we have $\Pscr\otimes \pi_X^*\Fscr\in D^b_{\coh}(X\times_S Y)$. Then first of all $\Pscr = \Pscr\otimes \pi_X^*\Ol_X \in D^b_{\coh}(X\times_S Y)$.

By \cite[Proposition 4.3]{Illusie}, it is enough to show that for every point $ (x,y)\in X\times_S Y$ there exist numbers $a,b$ such that
$\Pscr\in D_{\coh}^{[a,b]}(X\times_S Y)$ in a neighborhood of $(x,y)$ and
$$\text{tor amp}_{\Ol_{X,x}}(\Pscr_{(x,y)})\subset [a,b],$$
where $x=\pi_X((x,y))\in X$.

Since $\Pscr_{(x,y)}\otimes_{\Ol_{X,x}}M_x$ is bounded for all $M_x$ (because we know that we land in $D^b_{\coh}(X\times_S Y)$), there exists some $N$ such that 
$$\Tor^{\Ol_{X,x}}_i(\Pscr_{(x,y)}, k(x))=0$$
for all $i\geq N$, where $k(x)$ is the residue field at $x$.

Now replace $\Pscr_{(x,y)}$ by a bounded above complex of free $\Ol_{X,x}$-modules
$$\Qscr_{(x,y)} := \ldots \to \Qscr_{(x,y)}^n \to \Qscr_{(x,y)}^{n+1} \to \ldots \to \Qscr_{(x,y)}^b \to 0.$$
The complex is exact in low enough degree since $\Pscr \in D^b_{\coh}(X\times_S Y)$. Let $a < -N$ be low enough so that $\Qscr_{(x,y)}$ is exact in degree $\leq a+1$, and truncate the complex by taking $\Kscr_{(x,y)}=\ker (\Qscr_{(x,y)}^{a+1}\to\Qscr_{(x,y)}^{a+2})$ so that we have a complex 
$$ 0\to \Kscr_{(x,y)} \to \Qscr_{(x,y)}^{a+1} \to \Qscr_{(x,y)}^{a+2} \to \ldots \to \Qscr_{(x,y)}^b \to 0$$
which is also quasi-isomorphic to $\Pscr_{(x,y)}$.

Now consider the distinguished triangle $\Kscr_{(x,y)}[-a-1] \to \Qscr^{\geq a+1} \to \Pscr$. By the long exact sequence on Tor we have 
$$\Tor^{\Ol_{X,x}}_1(\Kscr_{(x,y)},k(x)) = \Tor^{\Ol_{X,x}}_{-a}(\Kscr_{(x,y)}[-a-1],k(x))= \Tor^{\Ol_{X,x}}_{-a}(\Qscr_{(x,y)}^{\geq a+1},k(x))=0$$
since $\Qscr^{\geq a+1}$ is zero in degree $a$.

Hence we obtain that $\Kscr_{(x,y)}$ is also locally free as an $\Ol_{X,x}$-module, hence $\Pscr_{(x,y)}$ is quasi-isomorphic to a bounded complex of locally free $\Ol_{X,x}$-modules and $\text{tor amp}_{\Ol_{X,x}}(\Pscr_{(x,y)})\subset [a,b].$

%

The other implication is clear.
\end{proof}

\begin{remark}
The notion in Lemma \ref{finite homological dimension}(2) corresponds to the notion of ``finite homological dimension'' in \cite{spaniards1, spaniards2}. The notion of perfection with respect to a morphism might seem stronger at first sight, but as we just saw, the two notions are actually equivalent.
\end{remark}

\begin{corollary}\label{X-perfect}
Let $X$ be projective over $S$. Let $\Pscr \in D^b_{\Qch}(X\times_S Y)$. The functor $\Phi^{X\to Y}_{\Pscr}$ maps $D^b_{\coh}(X)$ to $D^b_{\coh}(Y)$ if and only if $\Pscr$ is in $D^b_{\coh}(X\times_S Y)$ and is $\pi_X$-perfect.
\end{corollary}

\begin{proof}
Assume that $\Phi^{X\to Y}_{\Pscr}$ maps $D^b_{\coh}(X)$ to $D^b_{\coh}(Y)$. By Lemma \ref{beilinson} it follows that the functor $\pi_X^*(-) \otimes \Pscr$ sends $D^b_{\coh}(X)$ to $D^b_{\coh}(X\times_S Y)$. We conclude by Lemma \ref{finite homological dimension}. The other implication is clear.
\end{proof}

\begin{corollary}\label{both_perfect}
Assume that $X\to S$ and $Y\to S$ are projective. Let $\Pscr \in D^b_{\Qch}(X\times_S Y)$. If the functor $\Phi^{X\to Y}_{\Pscr}$ gives an equivalence $D^b_{\coh}(X)$ to $D^b_{\coh}(Y)$, then if follows that $\Pscr$ is $\pi_X$- and $\pi_Y$-perfect. Hence in particular it has left and right adjoints given by the formulas of Theorems \ref{left} and \ref{right}.
\end{corollary}

\begin{proof}
$\Pscr$ is in $D^b_{coh}(X\times_S Y)$ by Lemma \ref{X-perfect}. By \cite[Proposition 2.10]{spaniards2} and Lemma \ref{finite homological dimension}, it follows that $\Pscr$ is $\pi_X$- and $\pi_Y$-perfect. We conclude by Theorems \ref{left} and \ref{right}.
\end{proof}

\end{document}